\documentclass[twoside,a4paper,10pt,reqno,intlimits]{siamltex}
\usepackage{amsmath,color}
\usepackage{epsfig}
\usepackage{amssymb}
\usepackage{amsfonts}
\usepackage{amsxtra}
\usepackage{xspace}
\usepackage{multirow}
\usepackage{euscript,mathrsfs}
\allowdisplaybreaks

\usepackage{enumitem}
\setenumerate{label={\rm (\alph{*})}}

\usepackage[frak]{paper_diening}

\newtheorem{remark}[equation]{Remark}

\numberwithin{equation}{section}

\newcommand{\Div}{\divergence}

\newcommand{\R}{\mathbb R}

\newcommand{\E}{\mathbb E}
\newcommand{\p}{\mathbb P}

\newcommand{\F}{\mathfrak F}

\DeclareMathOperator*{\argmin}{arg\,min}

\newcommand{\dd}{\mathrm d}
\newcommand{\dx}{\, \mathrm{d}x}

\newcommand{\ds}{\, \mathrm{d}\sigma}
\newcommand{\dt}{\, \mathrm{d}t}
\newcommand{\dxt}{\,\mathrm{d}x\, \mathrm{d}t}
\newcommand{\dxs}{\,\mathrm{d}x\, \mathrm{d}\sigma}
\newcommand{\dif}{\mathrm{d}}
\newcommand{\totdif}{\mathrm{D}}

\newcommand{\prst}{\mathbb{P}}
\newcommand{\stred}{\mathbb{E}}

\newcommand{\db}[1]{\textcolor[rgb]{0.00,0.00,1.00}{  #1}}

\title{Space-time approximation of stochastic $p$-Laplace type systems} 
\author{D. Breit\footnotemark[1] \and M. Hofmanov\'a\footnotemark[2] \and S.
  Loisel{}\footnotemark[1] }\date{}

\begin{document}

\maketitle

\renewcommand{\thefootnote}{\fnsymbol{footnote}}
\footnotetext[1]{Heriot-Watt University, Department of Mathematics, Riccarton,
  EH14 4AS Edinburgh, UK. (d.breit@hw.ac.uk,
  s.loisel@hw.ac.uk)}
  \footnotetext[2]{Fakult\"at f\"ur Mathematik, Universit\"at Bielefeld, D-33501 Bielefeld, Germany. (hofmanova@math.uni-bielefeld.de) Financial support by the German Science Foundation DFG via the Collaborative Research Center SFB1283 is gratefully acknowledged.}

\begin{abstract}\\
We consider systems of stochastic evolutionary equations of $p$-Laplace type. We establish convergence rates for a finite-element based space-time approximation, where the error is measured in a suitable quasi-norm. Under natural regularity assumptions on the solution, our main result provides linear convergence in space
and convergence of order $\alpha$ in time for all $\alpha\in(0,\frac{1}{2})$. The key ingredient of our analysis is a random time-grid, which allows us to compensate for the lack of time regularity.
Our theoretical results are confirmed by numerical experiments.\\


\end{abstract}

\date{\small \today}


\section{Introduction}
We study the space-time discretization of stochastic evolutionary PDEs of the type
\begin{align}\label{eq:}
\begin{cases}\dd\bfu&=\Div\bfS(\nabla\bfu)\dt+\Phi(\bfu)\,\dd W\quad\text{in}\quad (0,T)\times\mathcal O,\\
\,\,\,\bfu&=0\qquad\qquad\qquad\qquad\qquad\,\,\,\,\,\,\, \text{in}\quad (0,T)\times\partial\mathcal O,\\
\bfu(0)&=\bfu_0\qquad\qquad\qquad\qquad\qquad\,\,\, \text{in}\quad \mathcal O,\end{cases}
\end{align}
in a bounded Lipschitz domain $\mathcal{O}\subset\R^d$ and a finite time interval $(0,T)$, where the solution $\bfu$ takes values in $\R^D$ for some $D\in\mathbb{N}$. Here $\bfS$ is a nonlinear operator with $p$-growth (see \eqref{eq:Sp} for a precise definition), for instance the $p$-Laplacian
\begin{align}\label{eq:S}
\bfS(\nabla\bfu)=\big(\kappa+|\nabla\bfu|\big)^{p-2}\nabla\bfu,
\end{align}
where $p\in(1,\infty)$ and $\kappa\geq0$.
We assume that $W$ is a cylindrical Wiener process  in a Hilbert space defined on a filtered probability space $(\Omega,\mathfrak{F},(\mathfrak F_t)_{t\geq0},\p)$  and $\Phi$ has linear growth (see Section \ref{subsec:prob} for more details).
This system can be understood as a model for a large class of problems important for applications. We particularly mention the flow of
 non-Newtonian fluids (the literature devoted to the deterministic setting is very extensive, for the corresponding stochastic counterparts we refer the reader to \cite{Br,TeYo,Yo}).
 \subsection{The deterministic problem}
The deterministic analogue of \eqref{eq:}, which reads as
\begin{align}\label{eq:0}
\begin{cases}\partial_t\bfu&=\Div\bfS(\nabla\bfu)+\bff\quad\text{in}\quad (0,T)\times\mathcal O,\\
\,\,\,\bfu&=0\qquad\qquad\qquad\,\,\,\, \text{in}\quad (0,T)\times\partial\mathcal O,\\
\bfu(0)&=\bfu_0\qquad\qquad\qquad \text{in}\quad \mathcal O,\end{cases}
\end{align}
with a given function $\bff$, seems to be well understood. Existence of a unique weak solution follows from the classical monotone operator theory. Also the regularity of solutions is well-known, see \cite{DiFr,DuMiSt,LaUrSo,Wi}. 

For the numerical approximation
of \eqref{eq:0} one approximates the time-derivative by a difference quotient and solves in every (discrete) time-step a stationary problem. The latter one finally has to be approximated by a finite element method. In order to understand
the convergence properties of the scheme it proved beneficial to introduce the nonlinear quantity $$\bfF(\bfxi)=(\kappa+|\bfxi|)^{\frac{p-2}{2}}\bfxi,\quad\bfxi\in\R^{D\times d},$$
which linearizes the approximation error in a certain sense. Starting from the paper \cite{BaLi1} it has been known that the correct way to express the  error   is  through  the quasi-norm
\begin{align*}
\|\bfF(\nabla\bfu)-\bfF(\nabla\bfv)\|_{L^2(\mathcal{O})}^2\sim\int_{\mathcal{O}} (\kappa+|\nabla\bfu|+|\nabla\bfv-\nabla\bfu|)^{p-2}|\nabla\bfu-\nabla\bfv|^2\dx
\end{align*}
as a distance  of functions $\bfu,\bfv\in W^{1,p}(\mathcal{O})$.
This led to optimal convergence results for finite element based space-time approximations of \eqref{eq:0}
in \cite{DER} for all $1<p<\infty$ (suboptimal results have been obtained in \cite{BaLi2}). 

More precisely, it was shown that
\begin{align*}
\max_{1\leq m\leq M}\|\bfu(t_m)-\bfu_{h,m}\|_{L^2(\mathcal O)}^2+\tau\sum_{m=1}^M\|\bfF(\nabla\bfu(t_m))-\bfF(\nabla\bfu_{h,m})\|_{L^2(\mathcal O)}^2\leq\,c\,\big(h^2+\tau^2\big),
\end{align*}
 where $(\bfu_{h,m})_{m=1}^M$ denotes the discrete solution with space discretization parameter $h$, time-discretization $\tau=\frac{T}{M}$ and time points $t_m=\tau m $, $m=1,\dots, M$.
The constant $c$ depends on the geometry of $\mathcal O$, on $p$, as well as on some quantities involving $\bfu$ and arising from the following regularity properties:  
\begin{align}\label{eq:regt1}
\begin{aligned}
&\bfF(\nabla\bfu)\in L^2(0,T;W^{1,2}(\mathcal O)),\ 
\bfF(\nabla\bfu) \in W^{1,2}(0,T;L^2(\mathcal O)),\\ 
&\nabla\bfu\in L^\infty(0,T;L^{2}(\mathcal O)),\ \partial_t\bfu\in L^\infty(0,T;L^{2}(\mathcal O)).
\end{aligned}
\end{align}

\subsection{The linear stochastic problem}
Concerning the stochastic problem \eqref{eq:}, the literature dedicated to the regularity theory for linear SPDEs (e.g. \eqref{eq:S} with $p=2$) is quite extensive, we refer to \cite{Kr,KrRo1,KrRo2} and the references therein. One can show that
\begin{align}\label{eq:1602}
\bfu\in L^2(\Omega;C^\alpha([0,T];W^{k,2}(\mathcal{O})))\quad\forall\alpha\in(0,\tfrac{1}{2})
\end{align}
provided the initial datum is smooth and $\Phi$ satisfies appropriate assumptions (here, the order $k\in\mathbb N$ is mainly determined by the smoothness of $\Phi$).

There is also a growing literature on the numerical approximation of linear SPDEs, see e.g. \cite{LPS,Ya1,Ya2}. 
One possible way is an implicit Euler scheme. 
Given a finite dimensional subspace $V_{h,0}\subset W^{1,2}_0(\mathcal{O})$ and an initial datum $\bfu_{h,0}\in V_{h,0}$ one computes
$\bfu_{h,m}$ such that
\begin{align}\label{tdiscr0}
\begin{aligned}
\int_{\mathcal{O}}&\bfu_{h,m}\cdot\bfvarphi \dx +\tau\int_{\mathcal O}\nabla\bfu_{h,m}:\nabla\bfphi\dx\\
&\qquad=\int_{\mathcal{O}}\bfu_{h,m-1}\cdot\bfvarphi \dx+\int_{\mathcal{O}}\Phi(\bfu_{h,m-1})\,\Delta_m W\cdot \bfvarphi\dx,\quad m=1,\dots, M,
\end{aligned}
\end{align}
for every $\bfphi\in V_{h,0}$, where $\Delta_m W=W(t_m)-W(t_{m-1})$.
(Note that in the case of \eqref{eq:}, \eqref{eq:S} with $p\neq2$ the corresponding scheme replacing  \eqref{tdiscr0} requires the solution of a nonlinear problem and is called a ``semi-algorithm''.
One must specify how the nonlinear problem is solved to obtain a concrete algorithms, see Section \ref{sec:num}.)

Based on the regularity \eqref{eq:1602} with some $k\in\mathbb{N}$ one can show that the approximation via \eqref{tdiscr0}
is of order one with respect to the space-discretization and of order $\alpha$ (for all $\alpha<\frac{1}{2}$) with respect to the time-discretization. The appropriate metric to quantify the error in is
\begin{align*}
\E\bigg[\max_{1\leq m\leq M}\int_{\mathcal{O}}|\bfu(t_m)-\bfu_{h,m}|^2\dx+\tau\sum_{m=1}^M\int_{\mathcal{O}}|\nabla\bfu(t_m)-\nabla\bfu_{h,m}|^2\dx\bigg].
\end{align*}

\subsection{The nonlinear stochastic problem}
\label{subsec:1.2}
For nonlinear stochastic problems like \eqref{eq:} with $p\neq2$, there is a bulk of literature regarding the existence of analytically weak solutions. The popular variational approach by Pardoux \cite{Pa} provides an existence theory for a quite general class of stochastic evolutionary equations. Existence of (unique) analytically strong solutions to generalized $p$-Laplace stochastic PDEs with $p\geq 2$ has been proved in \cite{G}.
Regularity results for \eqref{eq:} with $1<p<\infty$  have been shown in \cite{Br3}. In particular,  the first author proves that the unique weak solution satisfies\footnote{Here $L^2_{w^*}(\Omega;L^\infty(0,T;L^2(\mathcal{O})))$ is the space of weakly$^*$-measurable mappings $h:\Omega\to L^\infty(0,T;L^2(\mathcal{O}))$ such that $\E\esssup_{0\leq t\leq T}\|h\|^2_{L^2(\mathcal{O})}<\infty.$}
\begin{align}\label{eq:regx}
\nabla\bfu\in L^2_{w^*}(\Omega;L^\infty(0,T;L^2(\mathcal{O}))),\quad\bfF(\nabla\bfu)\in L^2(\Omega;L^2(0,T;W^{1,2}(\mathcal{O}))),
\end{align}
at least locally in space
and hence $\bfu$ is a strong solution (in the analytical sense, see Definition~\ref{def:strong}). 

H\"older-regularity in time of $\bfu$ (with values in some $L^p$-space) can be shown directly from the equation once having established the existence of second derivatives. For linear SPDEs, a boot-strap argument yields the same for all spatial derivatives of the solution and so \eqref{eq:1602} holds. Both arguments do not apply for $\bfF(\nabla\bfu)$ if $p\neq2$. 
In fact, the best one can hope for is
\begin{align}\label{eq:regt}
\bfu\in L^2(\Omega;C^{\alpha}([0,T];L^2(\mathcal{O}))),\quad \bfF(\nabla\bfu)\in L^2(\Omega;W^{\alpha,2}(0,T;L^2(\mathcal{O}))),
\end{align}
for all $\alpha\in(0,\tfrac{1}{2})$, cf. \cite{BrHo}. 

The time-regularity \eqref{eq:regt} is obviously much lower than \eqref{eq:regt1} in the deterministic case, due to the roughness of the driving Wiener process. It is also much lower than \eqref{eq:1602} in the linear stochastic case, due to the limited space-regularity in the nonlinear setting. In fact, the mapping $t\mapsto\bfF(\nabla\bfu(t))$ is not expected to be continuous in time (with values in $L^2(\mathcal O)$).
Hence, the natural error quantity
\begin{align*}
\E\bigg[\max_{1\leq m\leq M}\int_{\mathcal{O}}|\bfu(t_m)-\bfu_{h,m}|^2\dx+\tau\sum_{m=1}^M\int_{\mathcal{O}}|\bfF(\nabla\bfu(t_m))-\bfF(\nabla\bfu_{h,m})|^2\dx\bigg],
\end{align*}
measuring the distance between the solution $\bfu$ and a finite-element based space-time approximation $\bfu_{h,m}$ solving a nonlinear version of \eqref{tdiscr0}, is not well-defined. The same problem appears in the deterministic case if the time
regularity of $\bfF(\nabla\bfu)$ is too low as a consequence of irregular data, cf. \cite{BrMe}.

In order to overcome this problem we introduce a random time-grid. To be precise, we consider random time-points $\mathfrak t_m$ which are  distributed uniformly in $[m\tau-\tau/4,m\tau+\tau/4]$, where $\tau=T/M$. These time points are defined on a probability space $(\hat\Omega,\hat\F,\hat\p)$ which is possibly different from $(\Omega,\F,\p)$, the space where the Brownian motion in \eqref{eq:} is defined. We obtain a corresponding error estimator by taking the expectation $\hat\E$ with respect to $(\hat\Omega,\hat\F,\hat\p)$.
Based on the regularity in \eqref{eq:regx} and \eqref{eq:regt} 
for some $\alpha\in(0,\tfrac{1}{2})$,
 we prove that
\begin{align}
\nonumber
\hat\E\otimes\E\bigg[\max_{1\leq m\leq M}\int_{\mathcal{O}}|\bfu(\mathfrak t_m)-\bfu_{h,m}|^2\dx
&+\sum_{m=1}^M\tau_{m}\int_{\mathcal{O}}|\bfF(\nabla\bfu(\mathfrak t_m))-\bfF(\nabla\bfu_{h,m})|^2\dx\bigg]\\&\leq\,c(h^2+\tau^{2\alpha}),\label{error:right}
\end{align}
where $\tau_{m}=\mathfrak{t}_{m}-\mathfrak{t}_{m-1}$, see Theorem \ref{thm:4}. In other words, we understand our scheme as a random variable on a product space $\hat\Omega\times\Omega$, where $\hat\omega\in\hat\Omega$ accounts for the randomness introduced through the random time-grid, whereas $\omega\in \Omega $ is the randomness coming from the Brownian motion in \eqref{eq:}. 
As a matter of fact, the convergence rates from the linear problem still hold: we have convergence order one in space and  $\alpha$ in time. This theoretical result is confirmed by numerical experiments in Section \ref{sec:num}. 

\subsection{Comparison with Gy\"ongy-Millet}
As space-time discretizations of evolutionary SPDEs with monotone coefficients were already studied in \cite{GM1,GM2}, let us explain what are the main differences with the result that we put forward in the present paper. The class of equations considered in \cite{GM1,GM2} is more general and contains \eqref{eq:} as a special case, at least if $p>\frac{2d}{d+2}$. Indeed, the approach in \cite{GM1,GM2} is based on a Gelfand triple $V\hookrightarrow H\hookrightarrow V^\ast$, which in our case of \eqref{eq:} corresponds to $H=L^2(\mathcal{O})$ and $V=W^{1,p}_0(\mathcal{O})$ as the main part of the equation given by $\Div\bfS(\nabla\bfu)$ takes values in $V^\ast$. Clearly, the situation $p<\tfrac{2d}{d+2}$ cannot be included as the embedding $W^{1,p}_0(\mathcal{O})\hookrightarrow L^2(\mathcal{O})$ fails, nevertheless, this drawback does not occur
in our approach.

The result of \cite{GM2} are convergence rates, similar to \eqref{error:right}, where the error is measured in discrete versions of the spaces
$L^\infty(0,T;H)$ and $L^2(0,T;V)$. 
In order to establish these rates, the authors suppose additional assumptions upon the regularity of the solution. These assumptions are rather restrictive and can in general not be expected from a solution to \eqref{eq:}. We will outline this in more detail now.
As far as the spatial regularity is concerned, it is assumed in \cite{GM2} that
\begin{align}\label{eq:Gy}
\sup_{0\leq t\leq T}\stred\|\bfu\|^2_{\mathcal H}+\stred\int_0^T\|\bfu(t)\|_{\mathcal V}^2\,\dif t\leq C
\end{align}
for some separable Hilbert spaces $\mathcal V$ and $\mathcal H$ with $\mathcal V\hookrightarrow\mathcal H\hookrightarrow V$.

For \eqref{eq:} with $p=2$ an obvious choice is given by $\mathcal V=W^{1,2}_0\cap W^{2,2}(\mathcal O)$ and $\mathcal H=W^{1,2}_0(\mathcal O)$ (this can still include fully nonlinear problems, which satisfy \eqref{eq:Sp} with $p=2$).
Note, that the spatial convergence rate in \cite{GM2} depends on the ``gap'' between $V$ and $\mathcal V$ and $H$ and $\mathcal H$ and this example leads to rate 1.

In the non-degenerate case $\kappa>0$, if $p>2$, one can
infer from \eqref{eq:regx} that \eqref{eq:Gy} holds with the choice $\mathcal V=W^{1,p}_0\cap W^{2,p}(\mathcal O)$ and $\mathcal H=W^{1,2}_0(\mathcal O)$. This is technically excluded in \cite{GM2}
since $\mathcal{V}$ is not a Hilbert space and $\mathcal H$ does not embed into $V=W^{1,p}_0(\mathcal O)$ any more (there might be some hope to remove these assumptions in \cite{GM2}).

If $p<2$, estimate \eqref{eq:Gy} only follows from
\eqref{eq:regx} if $\nabla\bfu$ is bounded. While such a regularity holds in the deterministic case under certain assumptions, cf. \cite{DiFr}, it can never be true uniformly in $
\omega$ in the stochastic situation due to the unboundedness of the driving Wiener process.

In the degenerate case $\kappa=0$, if $p>2$, it is not possible to infer any integrability of $\nabla^2\bfu$ from \eqref{eq:regx}, whereas one has
$\E\int_0^T\|\nabla^2\bfu\|^p_{L^p(\mathcal O)}\dt<\infty$
in the case $p<2$.
 Comparing this with \eqref{eq:Gy} it turns out that the degenerate case is completely excluded in \cite{GM2}.
 
Finally, the assumption on the time-regularity of the solution made in \cite{GM2} is comparable to \eqref{eq:1602} with $k=1$ which is much stronger than
\eqref{eq:regt} and can in general not be expected as already explained in Section \ref{subsec:1.2}.
Recall also that we use a random time-grid because of the lack of continuity of $\nabla \bfu$ in time.

Our proof of the convergence rate is not based on the assumption \eqref{eq:Gy}, which permits us to overcome the drawback explained above and to deal with general systems with $p$-growth. The method relies rather on space and time regularity of $\bfF(\nabla\bfu)$ which is the natural quantity, the regularity of which can be studied via energy methods. 
Let us finally mention that our convergence rate is the same as in \cite{GM2}. On the other hand, numerical experiments are not given in \cite{GM2}.

\subsection{Further bibliographical comments}
In the papers \cite{BD,BCP,CP} space-time discretizations of stochastic Navier--Stokes equations are considered. The paper \cite{BCP} contains a convergence analysis whereas in \cite{BD} and \cite{CP} convergence rates -- similar to our results -- were shown for the two-dimensional space-periodic problem. It would be of great interest to combine this with the results in our paper and to study numerical approximations of generalized Newtonian fluids, done in the deterministic case in \cite{BDR}.

For completeness, we note that randomized numerical schemes in various settings have already appeared in the literature,  in particular in the context of ordinary differential equations, evolution equations with time-irregular coefficients, uncertainty quantification or modulated Schr\"odinger equations \cite{Assyr, EKKL,Scheichl,HKS20, JN09, KW2, KWa, KW}.

\subsection{Outline}
Our paper is organized as follows. In Section 2, we provide the necessary mathematical background. Section 3 gives a semi-algorithm for solving \eqref{eq:} along with our main result, Theorem \ref{thm:4}, which is an error estimate for this semi-algorithm. Section 4 describes how our semi-algorithm can be implemented on a real computer by way of convex optimization. Our numerical experiments confirm the sharpness of Theorem \ref{thm:4}.

\section{Mathematical framework}
\label{sec:framework}

We now give the precise assumptions on the system \eqref{eq:}. In particular, we specify the assumptions concerning the nonlinear tensor $\bfS$ in \eqref{eq:S}, introduce function spaces and the framework for stochastic integration, give a precise definition of the solution to \eqref{eq:} with its properties and present the finite element set-up.

\subsection{The nonlinear tensor}

We assume that $\bfS:\R^{D\times d}\rightarrow \R^{D\times d}$ is of class $C^0(\R^{D\times d})\cap C^1(\R^{D\times d}\setminus{\{0\}})$ and satisfies
\begin{align}\label{eq:Sp}
\lambda (\kappa+|\bfxi|)^{p-2}|\bfzeta|^2\leq \totdif\bfS(\bfxi)(\bfzeta,\bfzeta)\leq \Lambda(\kappa+|\bfxi|)^{p-2} |\bfzeta|^2
\end{align}
for all $\bfxi,\bfzeta\in\R^{D\times d}$ with some positive constants $\lambda,\Lambda$, some $\kappa\geq0$ and $p\in(1,\infty)$.
It is well known from the deterministic setting (and was already discussed in \cite{Br3} for the stochastic case) that an important role for the system \eqref{eq:S} with $\bfS$ satisfying \eqref{eq:Sp} is played by the function
$$\bfF(\bfxi)=(\kappa+|\bfxi|)^{\frac{p-2}{2}}\bfxi, \quad\bfxi\in\R^{D\times d}.$$
By $|\bfxi|$ we denote the Euclidean norm of a matrix $\bfxi\in\R^{D\times d}$ arising from the inner product $\bfxi:\bfzeta=\sum_{i=1}^D\sum_{j=1}^d\xi_{ij}\zeta_{ij}$.
The following two properties are essential for our analysis (see \cite{DieE08} and \cite{DiRu0} for a proof).
\begin{lemma}\label{lems:hammer}
Let $\bfS$ satisfy assumption \eqref{eq:Sp} for some $p\in(1,\infty)$.
\begin{enumerate}
\item For all $\bfxi,\bfeta\in \R^{D\times d}$ we have
\begin{align*}
|\bfF(\bfxi)-\bfF(\bfeta)|^2\sim \big(\bfS(\bfxi)-\bfS(\bfeta)\big):(\bfxi-\bfeta),
\end{align*}
where the constants hidden in $\sim$ only depend on $p$.\footnote{We write $f\sim g$ provided $f\leq cg\leq C f$ for some constants $c,\,C>0$.}
\item For every $\varepsilon>0$ and all $\bfxi,\bfeta,\bfzeta\in \R^{D\times d}$ it holds
\begin{align*}
\big|\big(\bfS(\bfxi)-\bfS(\bfeta)\big):(\bfxi-\bfzeta)\big|&\leq\,c(p,\varepsilon)\big(\bfS(\bfxi)-\bfS(\bfeta)\big):(\bfxi-\bfeta)\\&+\varepsilon\big(\bfS(\bfxi)-\bfS(\bfzeta)\big):(\bfxi-\bfzeta).\\
\end{align*}
\end{enumerate}
\end{lemma}

\subsection{Function spaces}
We denote as usual by $L^p(\mathcal O)$ and $W^{1,p}(\mathcal O)$ for $p\in[1,\infty]$ Lebesgue and Sobolev spaces over an open set $\mathcal O\subset\R^d$. We denote by $W^{1,p}_0(\mathcal O)$ the closure of the smooth and compactly supported functions in $W^{1,p}(\mathcal O)$. We do not distinguish in the notation for the function spaces between scalar- and vector-valued functions. However, vector-valued functions will usually be denoted in bold case.

For a separable Banach space $(X,\|\cdot\|_X)$ we denote by $L^p(0,T;X)$ the set of (Bochner-) measurable functions $u:(0,T)\rightarrow X$ such that the mapping $t\mapsto \|u(t)\|_{X}\in L^p(0,T)$. 
The set $C([0,T];X)$ denotes the space of functions $u:[0,T]\rightarrow X$ which are continuous with respect to the norm topology on $(X,\|\cdot\|_X)$. For $\alpha\in(0,1]$ we write
$C^\alpha([0,T];X)$ for the space of H\"older continuous functions.
Finally, we denote by $W^{\alpha,p}(0,T;X)$ the fractional Sobolev space in time, i.e. the subspace of the Bochner space $L^{p}(0,T;X)$
consisting of the functions having finite $W^{\alpha,p}(0,T;X)$-norm.

\subsection{The probability framework}
\label{subsec:prob}

Let $(\Omega,\F,(\F_t)_{t\geq0},\prst)$ be a stochastic basis with a complete, right-continuous filtration. The process $W$ is a cylindrical Wiener process, that is, $W(t)=\sum_{k\geq1}\beta_k(t) e_k$ with $(\beta_k)_{k\geq1}$ being mutually independent real-valued standard Wiener processes relative to $(\F_t)_{t\geq0}$ and $(e_k)_{k\geq1}$ a complete orthonormal system in a sepa\-rable Hilbert space $\mathfrak{U}$.
To give the precise definition of the diffusion coefficient $\varPhi$, consider $\bfz\in L^2(\mathcal O)$ and let $\,\varPhi(\bfz):\mathfrak{U}\rightarrow L^2(\mathcal{O})$ be defined by $\Phi(\bfz)e_k=\bfg_k(\cdot,\bfz(\cdot))$. In particular, we suppose
that $\bfg_k\in C^1(\mathcal{O}\times\R^D)$ and the conditions
\begin{align}\label{eq:phi}
\sum_{k\geq1}|\bfg_k(x,\bfxi)|^2 \leq c(1+|\bfxi|^2)&,\qquad
\sum_{k\geq1}|\nabla_{\bfxi} \bfg_k(x,\bfxi)|^2\leq c,\\
\sum_{k\geq1}|\nabla_x \bfg_k(x,\bfxi)|^2 &\leq c(1+|\bfxi|^2),\label{eq:phi2}
\end{align}
hold for all $x\in \mathcal{O}$ and $\bfxi\in\R^D$.

The conditions imposed on $\varPhi$, particularly the first assumption from \eqref{eq:phi}, imply that
$\varPhi:L^2(\mathcal{O})\rightarrow L_2(\mathfrak{U};L^2(\mathcal{O})),$
where $L_2(\mathfrak{U};L^2(\mathcal{O}))$ denotes the collection of Hilbert-Schmidt operators from $\mathfrak{U}$ to $L^2(\mathcal{O})$. Thus, given a progressively measurable process $\bfu\in L^2(\Omega;L^2(0,T;L^2(\mathcal{O})))$, the stochastic integral in \eqref{eq:}
is a well-defined process taking values in $L^2(\mathcal{O})$ (see \cite{PrZa} for a detailed construction).
Moreover, we can multiply by test functions to obtain
 \begin{align*}
\bigg\langle\int_0^t \varPhi(\bfu)\,\dd W,\bfphi\bigg\rangle=\sum_{k\geq 1} \int_0^t\langle \bfg_k(\bfu),\bfphi\rangle\,\dd\beta_k,\quad \bfphi\in L^2(\mathcal O).
\end{align*}

The initial datum may be random in general, i.e. $\F_0$-measurable, and we assume at least $\bfu_0\in L^2(\Omega;L^2(\mathcal{O}))$.

\subsection{The concept of solution and preliminary results}
\label{subsec:solution}

In this subsection we recall the definition of weak and strong solution as well as the basic existence, uniqueness and regularity results established in \cite{Br3}.

\begin{definition}[Weak solution]
\label{def:weakp}
An $(\F_t)$-progressively measurable function
$$\bfu\in L^2(\Omega;C([0,T];L^2(\mathcal{O})))\cap L^p(\Omega; L^p(0,T;{W}_0^{1,p}(\mathcal{O})))$$
is called a weak solution to \eqref{eq:} with $\bfS$ satisfying \eqref{eq:Sp} if
for every $\bfphi\in C^\infty_c(\mathcal{O})$ and all $t\in[0,T]$ it holds true $\p$-a.s.
\begin{align*}
\int_{\mathcal{O}}\bfu(t)\cdot\bfvarphi\dx &+\int_0^t\int_{\mathcal{O}}\bfS(\nabla \bfu(\sigma)):\nabla\bfphi\dxs\\&=\int_{\mathcal{O}}\bfu_0\cdot\bfvarphi\dx+\int_{\mathcal{O}}\int_0^t\Phi(\bfu)\,\dd W\cdot \bfvarphi\dx.
\end{align*}
\end{definition}

\begin{definition}[Strong solution]
\label{def:strong}
A weak solution to \eqref{eq:} with $\bfS$ satisfying \eqref{eq:Sp} is called a strong solution provided
$$\Div\bfS(\nabla \bfu)\in L^1(\Omega;L^1_{\text{loc}}((0,T)\times\mathcal O))$$
and we have for all $t\in[0,T]$ $\mathbb{P}$-a.s.
\begin{align*}
\bfu(t)&=\bfu_0+\int_0^t\Div\,\bfS(\nabla \bfu)\ds+\int_0^t\Phi(\bfu)\,\dd W.
\end{align*}
\end{definition}

The following result, which is taken from \cite[Theorem 4, Theorem 5]{Br3}, is the starting point of our analysis. We remark that the case $\kappa=0$ is not included in \cite{Br3} but can be obtained by approximation.

\begin{theorem}\label{thm:prelim}
Assume $\bfu_0\in L^2(\Omega;W_0^{1,2}(\mathcal{O}))$ is an $\mathfrak F_0$-measurable random variable. Suppose further that \eqref{eq:Sp}, \eqref{eq:phi} and \eqref{eq:phi2} hold.
Then there is a unique strong solution $\bfu$ to \eqref{eq:} which satisfies
\begin{align*}
\E\bigg[\sup_{t\in(0,T)}\int_{\mathcal{O}'}|\nabla\bfu(t)|^2\dx+\int_0^T\int_{\mathcal{O}'}|\nabla\bfF(\nabla\bfu)|^2\dxt\bigg]<\infty
\end{align*}
for all $\mathcal{O}'\Subset \mathcal{O}$.
\end{theorem}

It is expected that the result above holds even globally in space provided the boundary of $\mathcal O$ is smooth enough - {at least under appropriate compatibility conditions such as zero boundary values for the noise (see the counterexamples by Krylov \cite{Kr} on the global regularity for SPDEs in connection with this)}. A proof is, however, still missing.

\subsection{Discretization in space}
\label{sec:Vh}

Let $\mathcal{O} \subset \setR^d$ be a connected, open
domain with polyhedral boundary. We assume that $\partial \mathcal{O}$ is
Lipschitz continuous. For a bounded (non-empty) set $U \subset
\setR^d$ we denote
by $h_U$ the diameter of $U$, and by $\rho_U$ the supremum of the
diameters of inscribed balls.  
We denote by $\mathscr{T}_h$ a simplicial subdivision of $\mathcal{O}$ with
\begin{align*}
  h &= \max_{\mathcal S \in \mathscr{T}_h} h_\mathcal S
\end{align*}
which is
non-degenerate in the sense that there is $\gamma_0>0$ such that
\begin{align}
  \label{eq:nondeg}
  \max_{\mathcal S \in \mathscr{T}_h} \frac{h_\mathcal S}{\rho_\mathcal S} \leq \gamma_0.
\end{align}
For $\mathcal S \in \mathscr{T}_h$ we define the set of neighbours $N_\mathcal S$ and
the neighbourhood $\mathcal M_\mathcal S$ by
\begin{align*}
  N_\mathcal S &:= \set{\mathcal S' \in \mathscr{T}_h \,:\, \overline{\mathcal S'} \cap
    \overline{\mathcal S} \not= \emptyset},
  \quad
  \mathcal M_\mathcal S := \text{interior} \bigcup_{\mathcal S' \in N_\mathcal S} \overline{\mathcal S'}.
\end{align*}
Note that for all $\mathcal S,\mathcal S' \in \mathscr{T}_h$: $\mathcal S' \subset
\overline{\mathcal M_{\mathcal S}} \Leftrightarrow \mathcal S \subset \overline{\mathcal M_{\mathcal S'}}
\Leftrightarrow \overline{\mathcal S} \cap \overline{\mathcal S'} \not=\emptyset$. Due
to our assumption on $\mathcal{O}$ the neighbourhoods $\mathcal M_\mathcal S$ are connected, open bounded sets. 

It is easy to see that the non-degeneracy~\eqref{eq:nondeg} of
$\mathscr{T}_h$ implies the following properties, where the implicit constants
are independent of $h$:
\begin{enumerate}
\item \label{mesh:SK} $\abs{\mathcal M_\mathcal S} \sim \abs{\mathcal S}$ for all $\mathcal S \in
  \mathscr{T}_h$, where $|\cdot|$ denotes the Lebesgue measure of a set.
\item \label{mesh:NK} There exists $m_1 \in \setN$ such that $\# N_\mathcal S
  \leq m_1$ for all $\mathcal S \in
  \mathscr{T}_h$.
\end{enumerate}
  
For $\mathcal{O}\subset \setR ^d$ and $\ell\in \setN _0$ we denote by
$\mathscr{P}_\ell(\mathcal{O})$ the polynomials on $\mathcal{O}$ of degree less than or equal
to $\ell$. Moreover, we set $\mathscr{P}_{-1}(\mathcal{O}):=\set {0}$. Let us
characterize the finite element space $V_h$ by
\begin{align*}
  V_h &:= \set{\bfv \in W^{1,1}(\mathcal{O})\,:\, \bfv|_{\mathcal S}
    \in \mathscr{P}_1(\mathcal S)\,\,\forall \mathcal S\in \mathscr T_h}
\end{align*}
Furthermore, we define $V_{h,0}:=V_h\cap W^{1,1}_0(\mathcal O)$.
In order to interpolate between the discrete and continuous function spaces we work with the $L^2(\mathcal O)$-orthogonal projection onto $V_{h,0}$ which we denote by $\Pi_h$.
 Regarding stability/approximability we have
  \begin{align}
    \label{eq:stab'}
 \int_\mathcal O \Big|\frac{\bfv-\Pi_h \bfv}{h}\Big|^2\dx+ \int_\mathcal O \abs{\nabla \bfv-\nabla\Pi_h \bfv}^2\dx &\leq
    \,c \int_{\mathcal O} \abs{\nabla
      \bfv}^2\dx
  \end{align}
for all $v\in W^{1,2}_0(\mathcal O)$ uniformly in $h$. The estimate for the second term on the left-hand side is explictely states in \cite[equ. 4.12]{BDSW}. The
estimate for the first term follows since $\Pi_h\bfv$ belongs to a finite-dimensional function space (using also Poincar\'e's inequality and inverse estimates).
The following crucial estimate is shown in \cite[Theorem 4.5]{BDSW}.

\begin{lemma}
  \label{thm:app_V}
  Let $\mathcal{O} \subset \setR^d$ be a connected, open
domain with polyhedral boundary and $\mathscr{T}_h$ a simplicial subdivision of $\mathcal O$ which is non-degenerate. For all
  $\bfv \in W^{1,p}_0(\mathcal O)$ with $\bfF(\nabla \bfv)
  \in W^{1,2}(\mathcal O)$ we have
  \begin{align}
    \label{eq:app_V2}
    \int_{\mathcal O} \bigabs{\bfF (\nabla \bfv) - \bfF (\nabla \Pi_h
      \bfv)}^2 \dx &\leq c\, h^2\,
    \int_{\mathcal O} \bigabs{\nabla \bfF(\nabla \bfv)}^2\dx,
  \end{align}
  where $c$ depends only on $p$ and $\gamma_0$.
\end{lemma}

\section{Finite element based space-time approximation}

With the preparations from the previous section at hand, we are able to formulate our algorithm for the space-time approximation of \eqref{eq:}. We construct a random partition of $[0,T]$ with average mesh size $\tau=T/M$ as follows. Let $\mathfrak t_0=0$ and let $\mathfrak t_m$ for $m=1,\dots,M$ be independent random variables such that $\mathfrak t_m$ is distributed uniformly in $[m\tau-\tau/4,m\tau+\tau/4]$. We assume that the random variables $\mathfrak{t}_{m}$ are defined on a probability space $(\hat\Omega,\hat{\mathfrak{F}},\hat\p)$ and we denote by $\hat\E$ the corresponding expectation. Let $\tau_m=\mathfrak t_m-\mathfrak t_{m-1}$ and observe that $\tau_m$ is a random variable satisfying $\tau/2\leq \tau_m\leq 3\tau/2$.
Let $\bfu_{h,0}:=\Pi_h\bfu_0$ and for every $m\in\{1,\dots,M\}$ find  $\bfu_{h,m}\in L^2(\Omega;V_{h,0})$ such that
for every $\bfphi\in V_{h,0}$ it holds true $\p$-a.s.
\begin{align}\label{tdiscr}
\begin{aligned}
\int_{\mathcal{O}}&\bfu_{h,m}\cdot\bfvarphi \dx +\tau_m\int_{\mathcal{O}}\bfS(\nabla\bfu_{h,m}):\nabla\bfphi\dx\\
&\qquad=\int_{\mathcal{O}}\bfu_{h,m-1}\cdot\bfvarphi \dx+\int_{\mathcal{O}}\Phi(\bfu_{h,m-1})\,\Delta_mW\cdot \bfvarphi\dx,
\end{aligned}
\end{align}
where $\Delta_m W=W(\mathfrak t_m)-W(\mathfrak t_{m-1})$.

\subsection{Error analysis}
\label{sec:error}

In this subsection we establish convergence with rates of the above defined algorithm.
It is the main result of the paper. 
\begin{theorem}\label{thm:4}
Let $\bfu$ be the unique weak solution to \eqref{eq:} in the sense of Definition \ref{def:weakp}, where $\bfu_0\in L^2(\Omega,W^{1,2}_0(\mathcal O))$ is $\F_0$-measurable. Suppose that \eqref{eq:phi} holds.
Finally, assume that
\begin{align}\label{reg1}
\bfF(\nabla\bfu)&\in L^2(\Omega;L^2(0,T;W^{1,2}(\mathcal{O}))),\nabla\bfu\in L^2_{w^*}(\Omega;L^\infty(0,T;L^{2}(\mathcal{O}))),\\
\label{reg2}\bfF(\nabla\bfu)&\in L^2 (\Omega;W^{\alpha,2}(0,T;L^{2}(\mathcal{O}))),\ \bfu\in L^2(\Omega;C^{\alpha}([0,T];L^2(\mathcal O))),
\end{align}
where $\alpha\in(0,\tfrac{1}{2})$.
Then we have
\begin{align}
\nonumber
\hat\E\otimes\E\bigg[\max_{1\leq m\leq M}\int_{\mathcal{O}}|\bfu(\mathfrak t_m)-\bfu_{h,m}|^2\dx&+\sum_{m=1}^M \tau_m \int_{\mathcal{O}}|\bfF(\nabla\bfu(\mathfrak t_m))-\bfF(\nabla\bfu_{h,m})|^2\dx\bigg]\\&\leq \,c\,\big(h^2+\tau^{2\alpha}\big),
\label{eq:thm:4}
\end{align}
where $(\bfu_{h,m})_{m=1}^M$ is the numerical solution to \eqref{eq:} given by \eqref{tdiscr}.
\end{theorem}

\begin{remark}\label{rem:sharp}
\begin{itemize}
\item It is interesting to note that we do not need a condition between $h$ and $\tau$ as commonly used for the numerical approximation of nonlinear parabolic problems, see, for instance, \cite{DER}. This is due to the estimate for the error with respect to the $L^2(\mathcal O)$-projection from \cite{BDSW} stated in Theorem \ref{thm:app_V}.
\item Due to the random time-grid
we have a better control of the error between the continuous solutions and its projected version, see \eqref{eq:2201} below. This is different from the agrument in \cite{BDSW}
\end{itemize}
\end{remark}

The rest of this section is devoted to the proof of Theorem \ref{thm:4}.

\begin{proof}[Proof of Theorem \ref{thm:4}]
Due to \eqref{eq:stab'} and \eqref{reg1} we have
\begin{align*}
&\hat\E\otimes\E\bigg[\max_{1\leq m\leq M}\int_{\mathcal{O}}|\Pi_h\bfu(\mathfrak t_m)-\bfu(\mathfrak t_m)|^2\dx\bigg]\\&\leq ch^2
\hat\E\otimes\E\bigg[\max_{1\leq m\leq M}\int_{\mathcal{O}}|\nabla\bfu(\mathfrak t_m)|^2\dx\bigg]\leq ch^2
\E\bigg[\sup_{t\in(0,T)}\int_{\mathcal{O}}|\nabla\bfu(t)|^2\dx\bigg]\leq ch^2.
\end{align*}
Furthermore, it holds
\begin{align*}
&\hat\E\otimes\E\bigg[\sum_{m=1}^M \tau_m \int_{\mathcal{O}}|\bfF(\nabla\Pi_h\bfu(\mathfrak t_m))-\bfF(\nabla\bfu(\mathfrak t_m))|^2\dx\bigg]\\&\leq ch^2
\hat\E\otimes\E\bigg[\sum_{m=1}^M \tau_m \int_{\mathcal{O}}|\nabla\bfF(\nabla\bfu(\mathfrak t_m))|^2\dx\bigg]
\end{align*}
using \eqref{eq:app_V2}, where
\begin{align}\label{eq:2201}
\begin{aligned}
&\hat\E\otimes\E\bigg[\sum_{m=1}^M \tau_m\int_{\mathcal{O}}|\nabla\bfF(\nabla\bfu(\mathfrak t_m))|^2\dx\bigg]\\
&=\frac{2}{\tau}\int_{\tau-\tau/4}^{\tau+\tau/4} \xi\,\E\int_{\mathcal{O}}|\nabla\bfF(\nabla\bfu(\xi))|^2\dx\, \dif\xi\\
	&\qquad+\frac{4}{\tau^2}\sum_{m=2}^M \int_{(m-1)\tau-\tau/4}^{(m-1)\tau+\tau/4}\int_{m\tau-\tau/4}^{m\tau+\tau/4} (\xi-\zeta)\E\int_{\mathcal{O}}|\nabla\bfF(\nabla\bfu(\xi))|^2\dx\, \dif\xi\,\dif\zeta\\
&\leq  c \,\E\|\nabla\bfF(\nabla\bfu)\|_{L^2(0,\tau+\tau/4;L^2({\mathcal{O}}))}^2\\
	&\qquad+ \frac{c}{\tau}\sum_{m=2}^M\int_{(m-1)\tau-\tau/4}^{(m-1)\tau+\tau/4} \dif\zeta \,\E\|\nabla\bfF(\nabla\bfu)\|_{L^2(m\tau-\tau/4,m\tau+\tau/4;L^2({\mathcal{O}}))}^2\\
&\leq c\,\E\|\nabla\bfF(\nabla\bfu)\|_{L^2(0,T;L^2({\mathcal{O}}))}^2\leq\,c
\end{aligned}
\end{align}
by \eqref{reg1}.
Consequently, the claim follows if we can show
\begin{align}\label{eq:thm:4'}
\begin{aligned}
\hat\E&\otimes\E\bigg[\max_{1\leq m\leq M}\int_{\mathcal{O}}|\Pi_h\bfu(\mathfrak t_m)-\bfu_{h,m}|^2\dx\bigg]\\&+\hat\E\otimes\E\bigg[\sum_{m=1}^M \tau_m \int_{\mathcal{O}}|\bfF(\nabla\Pi_h\bfu(\mathfrak t_m))-\bfF(\nabla\bfu_{h,m})|^2\dx\bigg]\leq \,c\,\big(h^2+\tau^{2\alpha}\big).
\end{aligned}
\end{align}
Since the proof of \eqref{eq:thm:4'} is rather long we split it into several parts as follows:
\begin{itemize}
\item Step 1: We consider the equation for the error $\bfe_m=\Pi_h\bfu(\mathfrak t_m)-\bfu_{h,m}$ and test it with the error itself. As a result we express the error from time-step $m$ in terms the error
from time-step $m-1$.
\item Step 2: We apply the first estimates and iterate the expression from Step 1 which yields an estimate for the error quantity on the left-hand side of \eqref{eq:thm:4}.
\item Step 3: We estimate the stochastic integral in the estimate from Step 2.
\item Step 4: We estimate the error arising from the time-discretization. This crucially depends on the fractional differentiability in time which we assume for the continuous solution, cf. \eqref{reg2}, and
the averaging property of the expectation with respect to $(\hat\Omega,\hat{\mathfrak{F}},\hat{\mathbb P})$. 
\end{itemize}

\textbf{Step 1:} Subtracting \eqref{tdiscr} from the weak formulation of \eqref{eq:} we obtain
\begin{align*}
\begin{aligned}
\int_{\mathcal{O}}&\bfe_m\cdot\bfvarphi \dx +\int_{\mathfrak t_{m-1}}^{\mathfrak t_m}\int_{\mathcal{O}}\Big(\bfS(\nabla\bfu)-\bfS(\nabla\bfu_{h,m})\Big):\nabla\bfphi\dxs\\&=\int_{\mathcal{O}}\bfe_{m-1}\cdot\bfvarphi \dx+\int_{\mathcal{O}}\bigg(\int_{\mathfrak t_{m-1}}^{\mathfrak t_m}\Phi(\bfu)\,\dd W-\Phi(\bfu_{h,m-1})\,\Delta_mW\bigg)\cdot \bfvarphi\dx
\end{aligned}
\end{align*}
for every $\bfphi\in V_{h,0}$, or, equivalently,
\begin{align*}
\begin{aligned}
\int_{\mathcal{O}}&\bfe_m\cdot\bfvarphi \dx +\tau_m\int_{\mathcal{O}}\Big(\bfS(\nabla\Pi_h\bfu(\mathfrak t_m))-\bfS(\nabla\bfu_{h,m})\Big):\nabla\bfphi\dx\\
&=\int_{\mathfrak t_{m-1}}^{\mathfrak t_m}\int_{\mathcal{O}}\Big(\bfS(\nabla\Pi_h\bfu(\mathfrak t_m))-\bfS(\nabla\bfu(\sigma))\Big):\nabla\bfphi\dxs+\int_\mathcal{O}\bfe_{m-1}\cdot\bfvarphi \dx\\
&+\int_{\mathcal{O}}\bigg(\int_{\mathfrak t_{m-1}}^{\mathfrak t_m}\Phi(\bfu)\,\dd W-\Phi(\bfu_{h,m-1})\,\Delta_mW\bigg)\cdot \bfvarphi\dx.
\end{aligned}
\end{align*}
Setting $\bfphi=\bfe_{m}$ using the identity $\bfa\cdot(\bfa-\bfb)=\frac{1}{2}\big(|\bfa|^2-|\bfb|^2+|\bfa-\bfb|^2\big)$ (which holds for any $\bfa,\bfb\in\mathbb R^n$, here with $\bfa=\bfe_m$ and $\bfb=\bfe_{m-1}$), we gain
\begin{align*}
&\int_{\mathcal{O}}\frac{1}{2}\big(|\bfe_m|^2-|\bfe_{m-1}|^2+|\bfe_m-\bfe_{m-1}|^2\big) \dx\\&\quad\quad+\tau_m\int_{\mathcal{O}}\Big(\bfS(\nabla\Pi_h\bfu(\mathfrak t_m))-\bfS(\nabla\bfu_{h,m})\Big):\nabla\bfe_m\dx\\
&\quad\quad
=\tau_m\int_{\mathcal{O}}\Big(\bfS(\nabla\Pi_h\bfu(\mathfrak t_m))-\bfS(\nabla\bfu(\sigma))\Big):\nabla\bfe_{m}\dxs\\
&\quad\quad+\int_{\mathcal{O}}\bigg(\int_{\mathfrak t_{m-1}}^{\mathfrak t_m}\Phi(\bfu)\,\dd W-\Phi(\bfu_{h,m-1})\,\Delta_mW\bigg)\cdot \bfe_{m}\dx=:\db{I_{m,1}+I_{m,2}}.
\end{align*}
 \\
\textbf{Step 2:} Applying Lemma \ref{lems:hammer} we have
\begin{align*}
I_{m,1}&\leq \,\varepsilon\,\tau_m\int_{\mathcal{O}}\Big(\bfS(\nabla\Pi_h\bfu(\mathfrak t_m))-\bfS(\nabla\bfu_{h,m})\Big):\nabla(\Pi_h\bfu(\mathfrak t_m)-\bfu_{h,m})\dxs\\
&+c_\varepsilon \tau_m\int_{\mathcal{O}}\Big(\bfS(\nabla\Pi_h\bfu(\mathfrak t_m))-\bfS(\nabla\bfu({\sigma}))\Big):\nabla(\Pi_h\bfu(\mathfrak t_m)-\bfu({\sigma}))\dxs\\
&\leq\,c\,\varepsilon\,\tau_m\int_{\mathcal{O}}|\bfF(\nabla\Pi_h\bfu(\mathfrak t_m))-\bfF(\nabla\bfu_{h,m})|^2\dx\\&+c_\varepsilon\,\tau_m\int_{\mathcal{O}}|\bfF(\nabla\Pi_h\bfu(\mathfrak t_m))-\bfF(\nabla\bfu(\mathfrak t_m))|^2\dxs
\\&+c_\varepsilon\,\tau_m\int_{\mathcal{O}}|\bfF(\nabla{\bfu(\mathfrak t_m)})-\bfF(\nabla\bfu({\sigma}))|^2\dxs
\end{align*} 
for every $\varepsilon>0$.
Plugging all together and choosing $\varepsilon$ small enough (to absorb the corresponding term to the left-hand side) and using again Lemma \ref{lems:hammer} (a) we have shown
\begin{align*}
\int_{\mathcal{O}}&|\bfe_m|^2 \dx +\frac{1}{2}\int_{\mathcal{O}}|\bfe_m-\bfe_{m-1}|^2 \dx+c\tau_m\int_{\mathcal{O}}|\bfF(\nabla{\Pi_{h}}\bfu(\mathfrak t_m))-\bfF(\nabla\bfu_{h,m})|^2\dx\\
&\leq\,c\, \int_{\mathfrak t_{m-1}}^{\mathfrak t_m}\int_{\mathcal{O}}|\bfF(\nabla\Pi_h\bfu(\mathfrak t_m))-\bfF(\nabla\bfu(\mathfrak t_m))|^2\dxs+\int_{\mathcal{O}}|\bfe_{m-1}|^2 \dx\\
&+\,c\,\int_{\mathfrak t_{m-1}}^{\mathfrak t_m}\int_{\mathcal{O}}|\bfF(\nabla\bfu(\mathfrak t_m))-\bfF(\nabla\bfu(\sigma))|^2\dxs\\
&+c\,\int_{\mathcal{O}}\bigg(\int_{\mathfrak t_{m-1}}^{\mathfrak t_m}\Phi(\bfu)\,\dd W-\Phi(\bfu_{h,m-1})\,\Delta_m W\bigg)\cdot \bfe_{m}\dx.
\end{align*}
By 
\eqref{eq:app_V2} we conclude
\begin{align*}
\int_{\mathcal{O}}&|\bfe_m|^2 \dx +\frac{1}{2}\int_{\mathcal{O}}|\bfe_m-\bfe_{m-1}|^2 \dx+c\,\tau_m\int_{\mathcal{O}}|\bfF(\nabla{\Pi_{h}}\bfu(\mathfrak t_m))-\bfF(\nabla\bfu_m)|^2\dx\\
&\leq \,c\,\int_{\mathfrak t_{m-1}}^{\mathfrak t_m}\int_{\mathcal{O}}|\bfF(\nabla\bfu(\mathfrak t_m))-\bfF(\nabla\bfu(\sigma))|^2\dxs+c\,\tau_mh^2\int_{\mathcal{O}}|\nabla\bfF(\nabla\bfu(\mathfrak t_m))|^2\dx\\
&+\,\int_{\mathcal{O}}|\bfe_{m-1}|^2 \dx+c\,\int_{\mathcal{O}}\bigg(\int_{\mathfrak t_{m-1}}^{\mathfrak t_m}\Phi(\bfu)\,\dd W-\Phi(\bfu_{h,m-1})\,\Delta_m W\bigg)\cdot \bfe_{m}\dx.
\end{align*}
Iterating this inequality yields
\begin{align*}
\int_{\mathcal{O}}&|\bfe_m|^2 \dx + \frac{1}{2}\sum_{n=1}^m \int_{\mathcal{O}}|\bfe_m-\bfe_{m-1}|^2 \dx+c\sum_{n=1}^m \tau_n\int_{\mathcal{O}}|\bfF(\nabla{\Pi_{h}}\bfu(\mathfrak t_n))-\bfF(\nabla\bfu_{h,n})|^2\dx\\
&\leq \,c\sum_{n=1}^m\int_{\mathfrak t_{n-1}}^{\mathfrak t_n}\int_{\mathcal{O}}|\bfF(\nabla\bfu(\mathfrak t_n))-\bfF(\nabla\bfu(\sigma))|^2\dxs+\,\int_{\mathcal{O}}|\bfe_{0}|^2 \dx\\
&+c\,\sum_{n=1}^m\int_{\mathcal{O}}\bigg(\int_{\mathfrak t_{n-1}}^{\mathfrak t_n}\Phi(\bfu)\,\dd W-\Phi(\bfu_{h,n-1})\,\Delta_nW\bigg)\cdot \bfe_{n}\dx\\
&+c\,h^2\sum_{n=1}^m\tau_n\int_{\mathcal{O}}|\nabla\bfF(\nabla\bfu(\mathfrak t_n))|^2\dx.
\end{align*}
We are left to estimate the stochastic integral (which we call $ \mathscr M_m$) and the error arising from the time-discretization (the first term on the right-hand side). This will be done in the following two steps.\\
\textbf{Step 3:}
In order to estimate the stochastic term we write
\begin{align*}
\mathscr M_m&=\db{\sum_{n=1}^m I_{n,2}}=
\sum_{n=1}^m\int_{\mathcal{O}}\int_{\mathfrak t_{n-1}}^{\mathfrak t_n}\big(\Phi(\bfu)-\Phi(\bfu_{h,n-1})\big)\,\dd W\cdot \bfe_{n}\dx\\
&= \sum_{n=1}^m\int_{\mathcal{O}}\int_{\mathfrak t_{n-1}}^{\mathfrak t_n}\big(\Phi(\bfu)-\Phi(\bfu_{h,n-1})\big)\,\dd W\cdot \bfe_{n-1}\dx\\
&+ \sum_{n=1}^m\int_{\mathcal{O}}\int_{\mathfrak t_{n-1}}^{\mathfrak t_n}\big(\Phi(\bfu)-\Phi(\bfu_{h,n-1})\big)\,\dd W\cdot (\bfe_n-\bfe_{n-1})\dx=:\mathscr M_{m,1}+\mathscr M_{m,2}.
\end{align*}
For $\mathscr M_{m,1}$ we gain by the Burgholder-Davis-Gundy inequality
\begin{align*}
&\E\bigg[\max_{1\leq m\leq M}\big|\mathscr M_{m,1}\big|\bigg]\leq\,c\,\E\bigg[\sum_{n=1}^M\int_{\mathfrak t_{n-1}}^{\mathfrak t_n}\|\Phi(\bfu)-\Phi(\bfu_{h,n-1})\|^2_{L_2(\mathfrak U,L^2({\mathcal{O}}))}\|\bfe_{n-1}\|^2_{L^2({\mathcal{O}})}\dt\bigg]^{\frac{1}{2}}\\
&\leq\,c\,\E\bigg[\max_{1\leq n\leq M}\|\bfe_{n}\|_{L^2({\mathcal{O}})}\bigg(\sum_{n=1}^M\int_{\mathfrak t_{n-1}}^{\mathfrak t_n}\|\Phi(\bfu)-\Phi(\bfu_{h,n-1})\|^2_{L_2(\mathfrak U,L^2({\mathcal{O}}))}\dt\bigg)^{\frac{1}{2}}\bigg]\\
&\leq\,\varepsilon\,\E\bigg[\max_{1\leq n\leq M}\|\bfe_{n}\|^2_{L^2({\mathcal{O}})}\bigg]+\,c_\varepsilon\,\E\bigg[\sum_{n=1}^M\int_{\mathfrak t_{n-1}}^{\mathfrak t_n}\|\bfu-\bfu_{h,n-1}\|_{L^2({\mathcal{O}})}^2\dt\bigg]\\
&\leq\,\varepsilon\,\E\bigg[\max_{1\leq n\leq M}\|\bfe_n\|^2_{L^2({\mathcal{O}})}\bigg]+\,c_\varepsilon\,\E\bigg[\sum_{n=1}^M\int_{\mathfrak t_{n-1}}^{\mathfrak t_n}\|\bfu-\bfu(\mathfrak t_{n-1})\|_{L^2({\mathcal{O}})}^2\dt\bigg]\\
&+\,c_\varepsilon\,\E\bigg[\sum_{n=1}^M\int_{\mathfrak t_{n-1}}^{\mathfrak t_n}\|\bfu(\mathfrak t_{n-1})-\Pi_h\bfu(\mathfrak t_{n-1})\|_{L^2({\mathcal{O}})}^2\dt\bigg]+\,c_\varepsilon\,\E\bigg[\sum_{n=1}^M\int_{\mathfrak t_{n-1}}^{\mathfrak t_n}\|\bfe_{n-1}\|_{L^2({\mathcal{O}})}^2\dt\bigg]
\end{align*}
Here, we also used \eqref{eq:phi} as well as Young's inequality for $\varepsilon>0$ arbitrary. Applying \eqref{eq:stab'} as well as \eqref{reg1} and \eqref{reg2} we gain
\begin{align*}
&\E\bigg[\max_{1\leq m\leq M}\big|\mathscr M_{m,1}\big|\bigg]\leq\,\varepsilon\,\E\bigg[\max_{1\leq n\leq M}\|\bfe_n\|^2_{L^2({\mathcal{O}})}\bigg]+\,c_\varepsilon\,\E\bigg[\sum_{n=1}^M \tau_{n}\|\bfe_{n-1}\|_{L^2({\mathcal{O}})}^2\bigg]\\
&+c_\varepsilon\tau^{2\alpha}\E\big[\|\bfu\|_{C^\alpha([0,T],L^2(\mathcal O))}^2\big]+c_\varepsilon h^2\E\bigg[\sup_{t\in(0,T)}\int_{\mathcal{O}}|\nabla\bfu(t)|^2\dx\bigg]\\
&\leq\,\varepsilon\,\E\bigg[\max_{1\leq n\leq M}\|\bfe_n\|^2_{L^2({\mathcal{O}})}\bigg]+\,c_\varepsilon\,\E\bigg[\sum_{n=1}^M \tau_{n}\|\bfe_{n-1}\|_{L^2({\mathcal{O}})}^2\bigg]+c_\varepsilon\tau^{2\alpha}+c_\varepsilon h^2.
\end{align*}
As far as $\mathscr M_{m,2}$ is concerned we argue similarly. Using Cauchy-Schwartz inequality, Young's inequality, It\^{o}-isometry and \eqref{eq:phi} we have
\begin{align*}
&\E\bigg[\max_{1\leq m\leq M}|\mathscr M_{m,2}|\bigg]\\&\leq \E\bigg[ \sum_{n=1}^M\bigg( \varepsilon \|\bfe_n-\bfe_{n-1}\|_{L^2({\mathcal{O}})}^2 +c_\varepsilon \left\| \int_{\mathfrak t_{n-1}}^{\mathfrak t_n}\big(\Phi(\bfu)-\Phi(\bfu_{h,n-1})\big)\,\dd W  \right\|_{L^2({\mathcal{O}})}^2\bigg) \bigg]\\
&\leq \varepsilon\E\bigg[ \sum_{n=1}^M \|\bfe_n-\bfe_{n-1}\|_{L^2({\mathcal{O}})}^2 \bigg] + c_\varepsilon\,\E\bigg[\sum_{n=1}^M\int_{\mathfrak t_{n-1}}^{\mathfrak t_n}\|\bfu-\bfu_{h,n-1}\|_{L^2({\mathcal{O}})}^2\dt\bigg]\\
&\leq \varepsilon\E\bigg[ \sum_{n=1}^M \|\bfe_n-\bfe_{n-1}\|_{L^2({\mathcal{O}})}^2 \bigg] +c_\varepsilon \,\E\bigg[\sum_{n=1}^M \int_{\mathfrak t_{n-1}}^{\mathfrak t_n}\|\bfu-\bfu(\mathfrak t_{n-1})\|_{L^2({\mathcal{O}})}^2\dt\bigg]\\&+c_\varepsilon \,\E\bigg[\sum_{n=1}^M \tau_{n}\|\bfu(\mathfrak t_{n-1})-\Pi_h\bfu(\mathfrak t_{n-1})\|_{L^2({\mathcal{O}})}^2\bigg]+\,c_\varepsilon\,\E\bigg[\sum_{n=1}^M \tau_n\|\bfe_{n-1}\|_{L^2({\mathcal{O}})}^2\bigg]\\
&\leq \varepsilon\E\bigg[ \sum_{n=1}^M \|\bfe_n-\bfe_{n-1}\|_{L^2({\mathcal{O}})}^2 \bigg] +\,c_\varepsilon\,\E\bigg[\sum_{n=1}^M \tau_{n}\|\bfe_{n-1}\|_{L^2({\mathcal{O}})}^2\bigg]+c_\varepsilon\tau^{2\alpha}+c_{\varepsilon}h^2
\end{align*}
as a consequence of \eqref{reg2}.
Hence we may apply the discrete Gronwall lemma, choose $\varepsilon$ sufficiently small and apply \eqref{eq:stab'} to $\bfe_0$ to deduce
\begin{align}\label{eq:2201b}
\begin{aligned}
\E\bigg[&\max_{1\leq m\leq M}\int_{\mathcal{O}}|\bfe_m|^2 \dx +\sum_{m=1}^M\tau_m\int_{\mathcal{O}}|\bfF(\nabla \Pi_h\bfu(\mathfrak t_m))-\bfF(\nabla\bfu_m)|^2\dx\bigg]\\
&\leq\,c\,\E\bigg[\sum_{m=1}^M\int_{\mathfrak t_{m-1}}^{\mathfrak t_m}\int_{\mathcal{O}}|\bfF(\nabla\bfu(\mathfrak t_m))-\bfF(\nabla\bfu(\sigma))|^2\dxs\bigg]+c\tau^{2\alpha}+ch^2\\
&+c\,h^2\,\E\bigg[\sum_{m=1}^M \tau_m\int_{\mathcal{O}}|\nabla\bfF(\nabla\bfu(\mathfrak t_m))|^2\dx\bigg].
\end{aligned}
\end{align}
It remains to estimate the first term on the right-hand side in the last step.\\
\textbf{Step 4:}
Now, we observe that due to the definition of the points $\mathfrak t_m$, $m=1,\dots,M-1,$ as independent uniformly distributed random variables, the expectation $\hat\E$ can be computed explicitly as follows
\begin{align*}
&\hat\E\otimes\E\bigg[\sum_{m=1}^M\int_{\mathfrak t_{m-1}}^{\mathfrak t_m}\int_{\mathcal{O}}|\bfF(\nabla\bfu(\mathfrak t_m))-\bfF(\nabla\bfu(\sigma))|^2\dxs\bigg]\\
&=\hat\E\otimes\E\bigg[\int_{0}^{\mathfrak t_1}\int_{\mathcal{O}}|\bfF(\nabla\bfu(\mathfrak t_1))-\bfF(\nabla\bfu(\sigma))|^2\dxs\bigg]\\
	&\quad+\sum_{m=2}^{M}\hat\E\otimes\E\bigg[\int_{\mathfrak t_{m-1}}^{\mathfrak t_m}\int_{\mathcal{O}}|\bfF(\nabla\bfu(\mathfrak t_m))-\bfF(\nabla\bfu(\sigma))|^2\dxs\bigg]\\
&=\frac{2}{\tau}\E\int_{\tau-\tau/4}^{\tau+\tau/4}\int_{0}^{\xi}\int_{\mathcal{O}}|\bfF(\nabla\bfu(\xi))-\bfF(\nabla\bfu(\sigma))|^2\dxs\, \dif\xi\\
	&\qquad+\frac{4}{\tau^2}\sum_{m=2}^{M}\int_{(m-1)\tau-\tau/4}^{(m-1)\tau+\tau/4}\E\int_{m\tau-\tau/4}^{m\tau+\tau/4}\int_{\zeta}^{\xi}\int_{\mathcal{O}}|\bfF(\nabla\bfu(\xi))-\bfF(\nabla\bfu(\sigma))|^2\dxs\, \dif\xi\,\dif\zeta\\
&\leq c\tau^{2\alpha} \E\|\bfF(\nabla\bfu)\|_{W^{\alpha,2}(0,\tau+\tau/4;L^2({\mathcal{O}}))}^2\\
	&\qquad+\frac{c\tau^{2\alpha}}{\tau}\sum_{m=2}^M \int_{(m-1)\tau-\tau/4}^{(m-1)\tau+\tau/4}\dif\zeta\,\E\|\bfF(\nabla\bfu)\|_{W^{\alpha,2}({(m-1)}\tau-\tau/4,m\tau+\tau/4;L^2({\mathcal{O}}))}^2\\
&\leq c\tau^{2\alpha}\E\|\bfF(\nabla\bfu)\|_{W^{\alpha,2}(0,T;L^2({\mathcal{O}}))}^2
\end{align*}
recalling the definition of the fractional Sobolev norm.
The last term in \eqref{eq:2201b} can be controlled by \eqref{eq:2201}.
Consequently,  in view of \eqref{reg1} and \eqref{reg2}, the proof of \eqref{eq:thm:4'} is complete and the claim follows.
\end{proof}

\section{Numerical experiments}
\label{sec:num}

\begin{figure}[t]
\includegraphics[width=6cm]{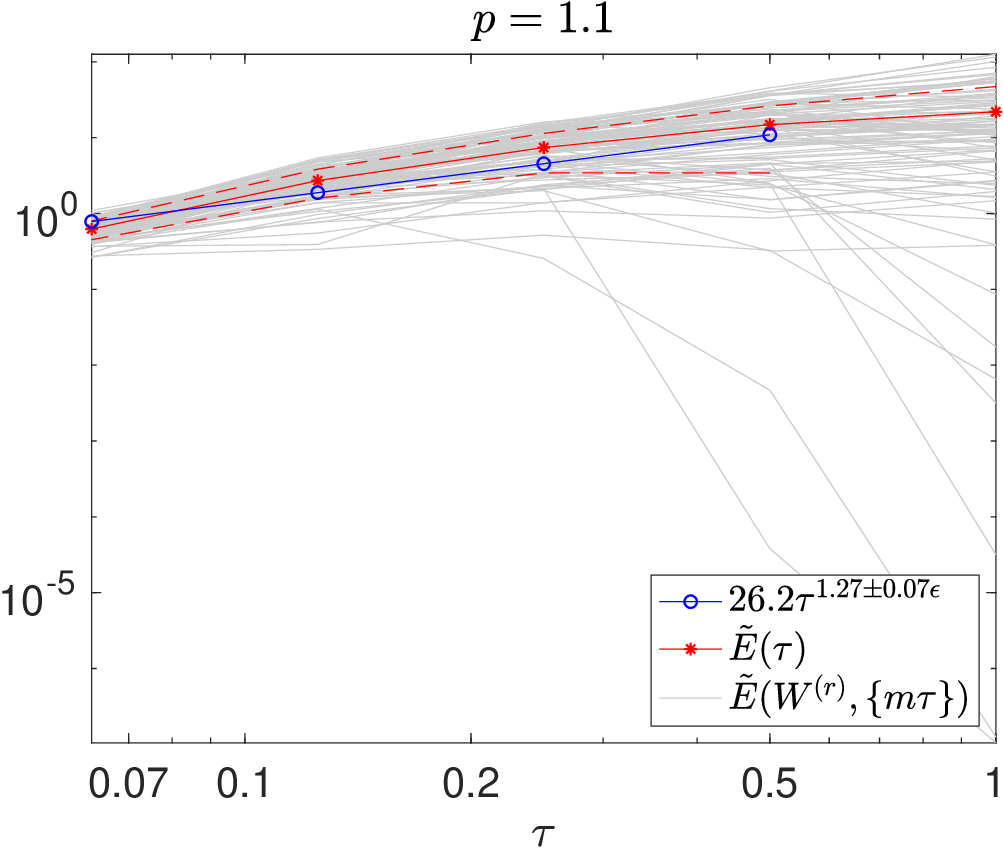}
\includegraphics[width=6cm]{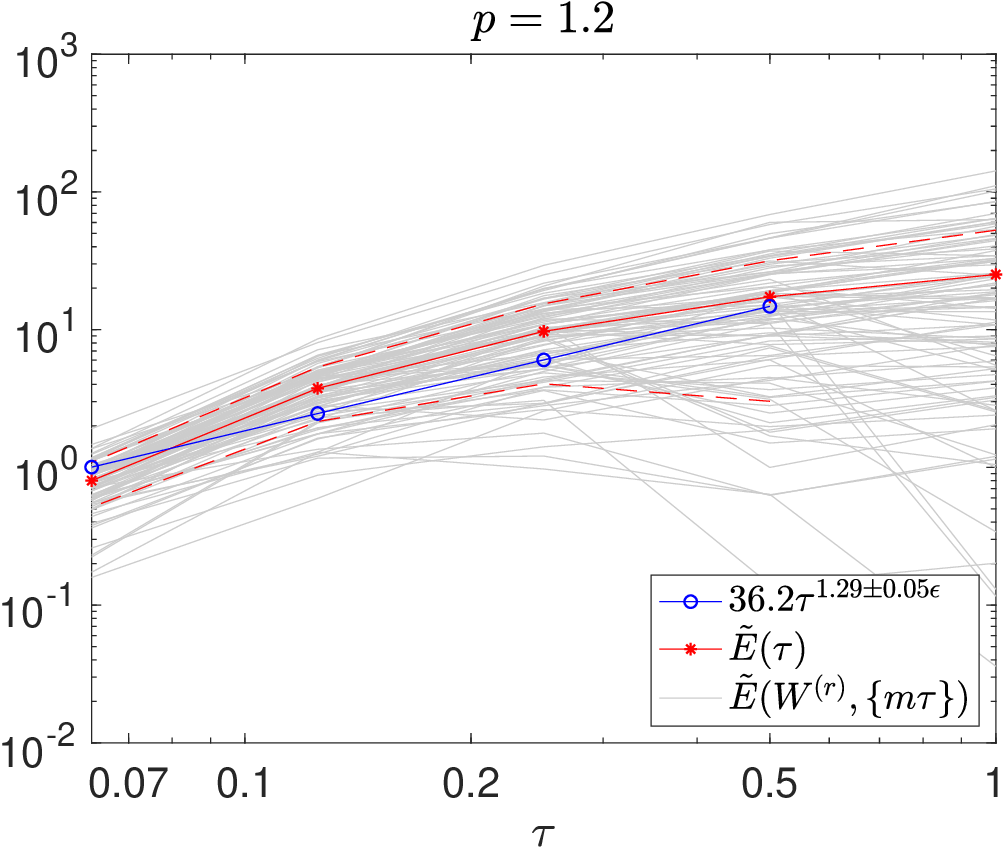} \\
\vspace{0.7cm}
\includegraphics[width=6cm]{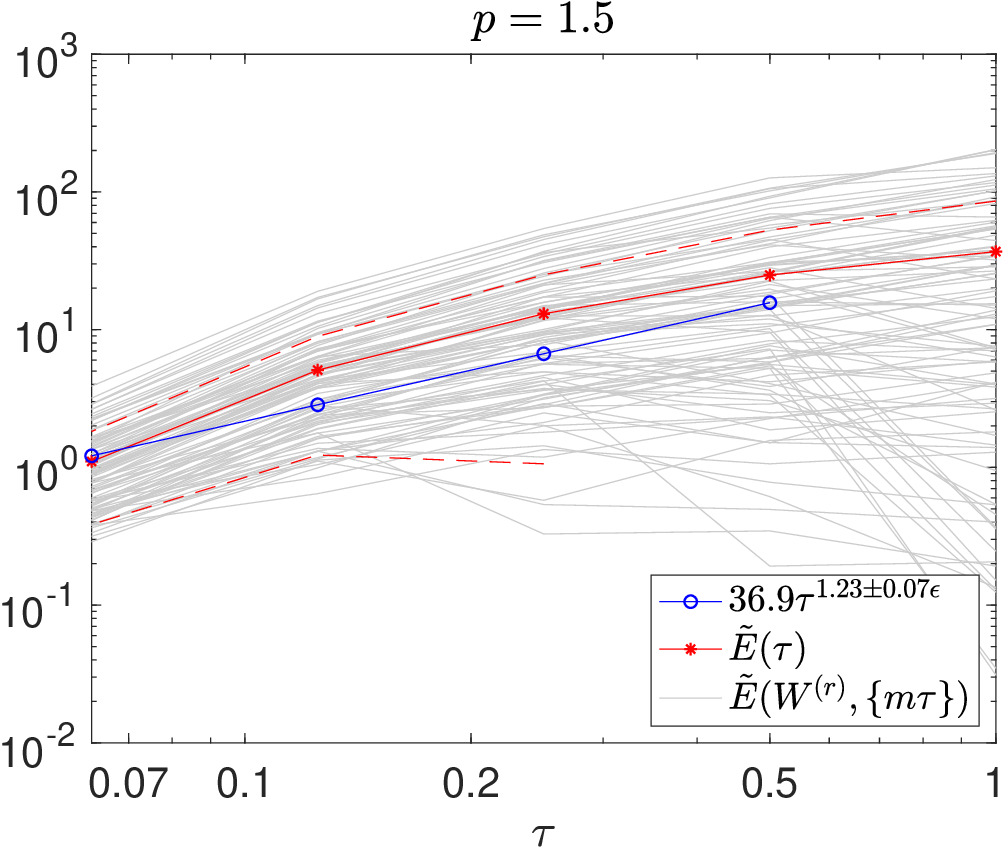}
\includegraphics[width=6cm]{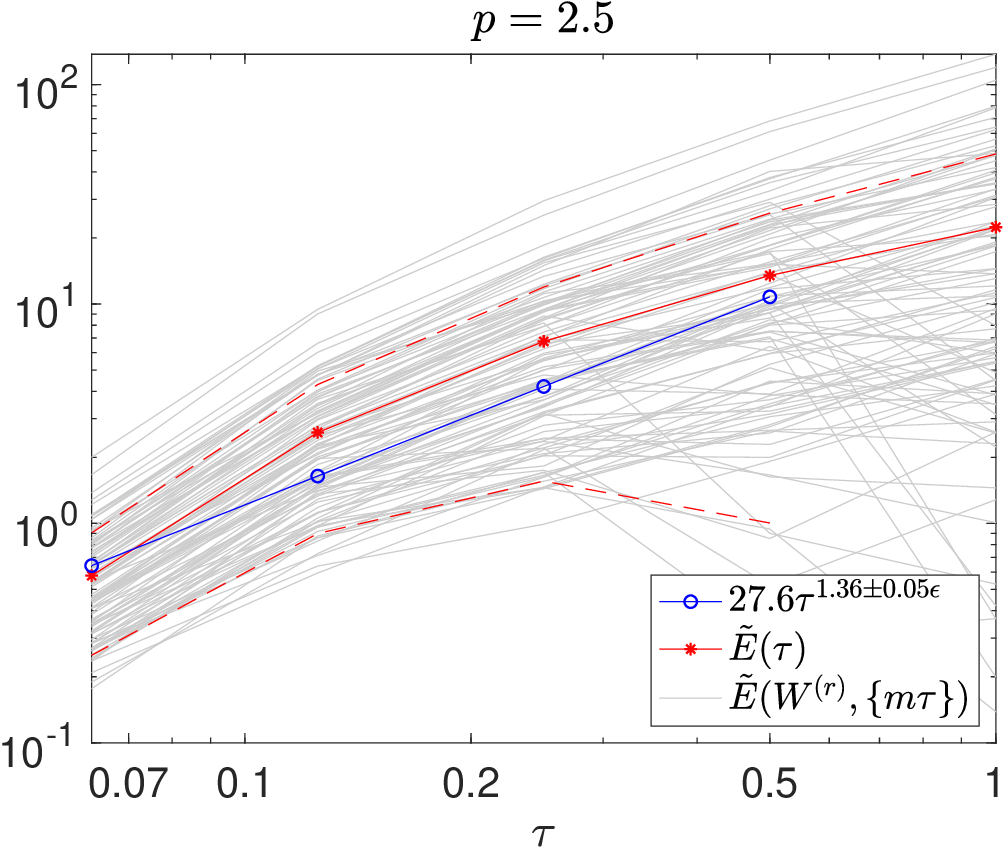}
\caption{Convergence: the function $\tilde{E}(\tau)$ of \eqref{e:errest} for $p \in \{1,1.2,1.5,2.5\}$.
The convergence rate $O(\tau^{1.3\pm0.1})$ is a biased estimator for $E(\tau)$; an unbiased estimator is $O(\tau^{0.88\pm0.14})$.\label{f:num}}
\end{figure}

Let $\mathcal{O}$ be a polygon, $\mathscr{T}_h$ is its simplicial mesh, and $V_h$ is the space of piecewise linear elements on $\mathscr{T}_h$ as introduced in Section \ref{sec:Vh}. We say that \eqref{tdiscr} describes a ``semi-algorithm'' for the numerical solution of the $p$-Laplacian, meaning that it does not indicate how the implicit equation can be solved for the unknown $\mathbf{u}_{h,m}$. Fortunately, this is a convex optimization problem, and the specific case of the $p$-Laplacian was recently analyzed in \cite{loisel2020}. We now briefly outline this procedure in the case $\kappa = 0$: If $\bfu_{h,m-1}$ and $\Delta_m W$ are given we have
\begin{align}
\mathbf{u}_{h,m} = \argmin_{\mathbf{u} \in V_h} \int_{\mathcal{O}} \Big({1 \over 2}|\mathbf{u}|^2 + {\tau_m \over p} |\nabla \mathbf{u}|^p
-
\overbrace{(\mathbf{u}_{h,m-1}+\Phi \Delta_m W)}^{\mathbf{f}} \mathbf{u}\Big)
\dx. \label{e:optversion}
\end{align}
Here, we have used the Dirichlet principle to rephrase the variational problem \eqref{tdiscr} in its elliptic energy-minimization form.
The quantity that we have denoted $\mathbf{f} = \mathbf{f}_{h,m}$ acts like a forcing for an elliptic $p$-Laplacian. Because $\mathbf{u}$ is piecewise linear,
one checks that the first two terms of the integrand are piecewise constant.
Recall further that $W$ does not depend on $x$ here. Hence, if we further assume that $\Phi$ is piecewise constant, then $\mathbf{f}$ and $\mathbf{u}$ are piecewise linear and so $\mathbf{f}\cdot\mathbf{u}$ is piecewise quadratic; thus, the integral in \eqref{e:optversion} can be computed exactly for any $\mathbf{u} \in V_h$. Because $\Phi$ is piecewise constant (and hence possibly discontinuous), the piecewise linear function $\mathbf{f}$ is likewise possibly discontinuous. Recall ``broken finite element'' space $\tilde{V}_h$ of possibly discontinuous piecewise linear functions on $\mathscr{T}_h$, we have that $\mathbf{f} \in \tilde{V}_h$.

Denote by $\{x^{(k)} \}$ the vertices of the triangulation $\mathscr{T}_h$, and $\{\mathcal{S}_j\}$ the simplices of $\mathscr{T}_h$. Denote by $\boldsymbol{\phi}_k$ the usual basis of $V_h$, so that $\mathbf{u} = \sum_k u_k \boldsymbol{\phi}_k$, where $u_k = \mathbf{u}(x^{(k)})$. Likewise, we denote by $\tilde{\boldsymbol{\phi}}_k$ the basis functions of the broken finite element space $\tilde{V}_h$.

Let $D^{(i)}$ be the discrete partial derivative matrix, defined by
$(D^{(i)})_{j,k} = {\partial \over \partial x_i}\boldsymbol{\phi}_k|_{\mathcal{S}_j}$. Note that this is well-defined because ${\partial \over \partial x_i}\boldsymbol{\phi}_k$ is constant on individual simplices $\mathcal{S}_j$.
After discretization by finite elements, one arrives at the following fully-discrete convex optimization problem:
\begin{align}
u_{h,m} = \argmin_{u} {1 \over 2}u^TPu + {\tau_m \over p} \bigg(\sum_{i,j} |\mathcal{S}_j|\, |(D^{(i)}u)_j|^p\bigg) - f^T\tilde{P}u; \label{e:discversion}
\end{align}
where $|\mathcal{S}_j|$ is the surface area of $\mathcal{S}_j$ and $P$ is the usual mass matrix. The matrix $\tilde{P}$ is the mass matrix for a broken finite element function $\mathbf{f} \in \tilde{V}_h$ and a function $\mathbf{u} \in V_h$. The boundary condition $\mathbf{u}|_{\partial \mathcal{O}}=0$ can be implemented via the linear constraints $u_k = 0$ for every $x^{(k)} \in \partial \mathcal{O}$. Practically speaking, one can avoid these extra constraints by making them ``explicit'' and writing $u = R^Tu_I$, where $u_I$ is $u$ restricted to the interior of $\mathcal{O}$, and $R^T$ is the relevant prolongation matrix to the closure of $\mathcal{O}$.

Problem \eqref{e:discversion} can be solved, e.g. by barrier methods. For simplicity, in the present paper, we use the the industrial convex programming solver MOSEK \cite{Andersen00} and CVX \cite{Grant2014}.

The goal of these experiments is to perform a numerical verification of our main estimate, given by \eqref{eq:thm:4}, with the main objective being of estimating the order $a = 2\alpha$ of the method in the variable $\tau$. In all of our experiments, we use $h \leq \tau$. Thus, the right-hand-side $R(\tau)$ of \eqref{eq:thm:4} satisfies $c\tau^a \leq R(\tau) \leq c\tau^a + O(\tau)$, and $O(\tau)$ is negligible compared to $c\tau^a$, because $a<1$ and $\tau$ is small. Thus, we seek to approximate the left-hand-side of \eqref{eq:thm:4} by $c\tau^a$ for some $c,a > 0$, and the regularity $\alpha$ will be estimated by 
\begin{align}
\alpha = a/2. \label{e:myalpha}
\end{align}

We now describe how we approximate the left-hand-side of \eqref{eq:thm:4}. We seek to measure the scaling in the $\tau$ variable, so instead of the notation $\mathbf{u}_{h,m}$, we use the notation $\mathbf{u}_{h}(\mathfrak{t}_m)$. If we have two time discretizations $\{\mathfrak{t}_m\} \subset \{\tilde{\mathfrak{t}}_m\}$, with corresponding solutions $\mathbf{u}_{h}(\mathfrak{t}_m)$ and $\mathbf{\tilde{u}}_{h}(\tilde{\mathfrak{t}}_m)$, then we say that $\mathbf{\tilde{u}}$ is a time-refinement of $\mathbf{u}$, since its time grid is finer. The error of $\mathbf{u}$ is $\mathbf{e}_h(\mathfrak{t}_m) = \mathbf{u}(\mathfrak{t}_m) - \mathbf{u}_{h}(\mathfrak{t}_m)$. Computing $\mathbf{u}_h$ and $\mathbf{e}_h$ exactly is equivalent to computing the true solution $\mathbf{u}$, so one cannot compute $\mathbf{e}$ exactly, and one must instead use an approximation. If $\mathbf{\tilde{u}}$ uses a finer time grid than $\mathbf{u}$, the quantity $\mathbf{\tilde{e}}_h(\mathfrak{t}_m) = \mathbf{\tilde{u}}_{h}(\mathfrak{t}_m) - \mathbf{u}_{h}(\mathfrak{t}_m)$ is an approximation of the error $\mathbf{e}_h$.

Denote by $E(W,\{\tau_m\})$ the quantity appearing inside the expectations on the left-hand-side of \eqref{eq:thm:4}, and $\tilde{E}$ the approximation of $E$ obtained by replacing $\mathbf{u}(\mathfrak{t}_m)$ by $\mathbf{\tilde{u}}_h(\mathfrak{t}_m)$. It is difficult to reconcile the random time-step distributions hypothesized in Theorem \ref{thm:4} with the requirement that $\{\tau_m\} \subset \{\tilde{\tau}_m\}$ so we use the deterministic times $\mathfrak{t}_m = m\tau$ and $\tilde{\mathfrak{t}}_m = m\tilde{\tau}$, instead of random time steps. Then, to ensure $\{\mathfrak{t}_m\} \subset \{\tilde{\mathfrak{t}}_m\}$, it suffices to ensure that $\tau$ is an integer multiple of $\tilde{\tau}$. Even though it was important for the analysis to have random time steps, we suspect this is only an issue as $\tau \to 0$. For our numerical experiments with finite $\tau$, our deterministic time-stepping strategy did not seem to have a deleterious effect on convergence.

With these deterministic time-steps, the left-hand-side of \eqref{eq:thm:4} equals to $E(\tau) = \mathbb{E}(E(W,\{m\tau\}))$. We estimate the expectation by Monte-Carlo experiments, i.e.
\begin{align}
\tilde{E}(\tau) & = {1 \over n_r}\sum_{r=1}^{n_r} \tilde{E}(W^{(r)},\{m\tau\}), \label{e:errest}
\end{align}
where $W^{(1)},\ldots,W^{(n_r)}$ are $n_r$ independent instances of the Wiener process $W$. We use $n_r = 100$ repetitions to compute $\tilde{E}(\tau)$.

We have sketched the results of these numerical experiments in Figure \ref{f:num}. We have solved the $p$-Laplacian for $p \in \{1.1,1.2,1.5,2.5\}$. The domain $\mathcal{O} = (0,1) \times (0,1)$ is the unit square and the grid size is fixed at $h=1/32$ and the time-step sizes vary in $\tau \in \{1,1/2,\ldots,1/32\}$. The initial data are $(u_{h,0})_k = 1$ for all $x^{(k)} \in \mathscr{T}_h$. The noise coefficient is $\Phi(x) = |x|^{-1/2}$, as approximated in a piecewise constant manner on $\mathscr{T}_h$. For each value of $p$, each instance $W^{(r)}$ of the Brownian motion $W$ gives rise to a function $\tau \to \tilde{E}(W^{(r)},\{m\tau\})$. There are $n_r = 100$ such functions for each value of $p$, plotted in light-gray. The averages of these functions correspond to $\tilde{E}(\tau)$ as per \eqref{e:errest}, plotted in red. The dashed red curves correspond to plus or minus one standard deviation away from $\tilde{E}(\tau)$.

We ran these experiments on a Macbook Pro, which takes about 10 hours to complete. Although Monte Carlo algorithms are ``embarrassingly parallel'', in practice we cannot run these experiments in parallel because of technical limitations in the third-party library CVX.

To estimate the convergence order $a$, we use linear regression to estimate the coefficients $c,a$ of $\tilde{E}(W^{(r)},\{m\tau\}) \approx c\tau^a$. 
For each value of $p$, we have sketched in Figure \ref{f:num} the corresponding line of regression as a solid blue line.
The notation $\tilde{c}h^{\tilde{a} \pm \sigma \epsilon}$ in the legend indicates that the estimated coefficient of linear regression is $\tilde{a}$, and that the standard deviation for this estimate $\tilde{a}$ of $a$ is $\sigma$ (e.g. $\epsilon$ is a standard normal random variable). Observe that across all values of $p$, we have approximately $\tilde{a} = 1.3 \pm 0.1$, hence using \eqref{e:myalpha}, we find $\tilde{\alpha} \approx 0.65 \pm 0.05$.
This is slightly larger than the maximum envisioned value $\alpha \approx 0.5$, which is probably caused by slight estimation bias, as we now describe.

\begin{lemma} \label{l:biascorr}
Let $0 < a \leq 1$.
Let $G(\tau)$ and $\tilde{G}(\tau)$ be functions that satisfy
\begin{align*}
G(\tau) & = \theta(\tau) \tau^a \quad\text{and}\quad \tilde{G}(\tau) = (G(\tau)-G(\tilde{\tau}))\eta(\tau,\tilde{\tau}),
\end{align*}
where $\theta(\tau)$ and $\eta(\tau,\tilde{\tau})$ are bounded nonzero continuous functions of $0 \leq \tilde{\tau} \leq \tau \leq 1$ with bounded derivatives $\theta'$ and ${\partial \over \partial \tau}\eta$. Then, for sufficiently small $0<\tilde{\tau}<\tau$,
\begin{align}
{\dd \over \dd\log \tau}\log G(\tau)  = a + O(\tau) \quad \text{and}\quad 
{\dd \over \dd\log \tau}\log \tilde{G}(\tau)  = 
(a+O(\tau))
\overbrace{
\left(
\tau^a \over \tau^a - \tilde{\tau}^a
\right)
}^{\beta_{\tau,\tilde{\tau}}(a)}. \label{e:biascorr}
\end{align}
We call $\beta = \beta_{\tau,\tilde{\tau}}(a)$ the {\bf bias}.
\end{lemma}

\begin{proof}
Direct calculation gives:
\begin{align*}
{\dd \over \dd\log \tau}\log G(\tau) & = a + {\theta'(\tau)\tau \over \theta(\tau)},
 \\
{\dd \over \dd\log \tau}\log \tilde{G}(\tau) & = 
{{\tau^a(\theta'(\tau)\tau + \theta(\tau)a)}
\over
\underbrace{\theta(\tau)\tau^a - \theta(\tilde{\tau}) \tilde{\tau}^a}_{B}
}
+
\tau{{\partial \eta \over \partial \tau}(\tau,\tilde{\tau}) \over \eta(\tau,\tilde{\tau})}.
\end{align*}
By the intermediate value theorem (IVT),
\begin{align*}
B & = \theta(\tau)(\tau^a-\tilde{\tau}^a)
-(\theta(\tau)-\theta(\tilde{\tau}))\tilde{\tau}^a
= \theta(\tau)(\tau^a-\tilde{\tau}^a)
-\theta'(\xi_1)(\tau-\tilde{\tau})\tilde{\tau}^a,
\end{align*}
where $\tilde{\tau} < \xi_1 < \tau$. The IVT applied to $\tau^a$ gives $a\xi_2^{a-1}(\tau-\tilde{\tau}) = \tau^a - \tilde{\tau}^a$ with $\tilde{\tau}<\xi_2<\tau$ and hence
\begin{align*}
B & = \theta(\tau)(\tau^a-\tilde{\tau}^a)\bigg(1-\frac{\theta'(\xi_1)}{\theta(\tau)a}\overbrace{\tilde{\tau}^a\xi_2^{1-a}}^{\xi_3}\bigg),
\end{align*}
where $\tilde\tau < \xi_3 < \tau$, because $\xi_3$ is a geometric mean. Finally, we use $1/(1-z) = 1+O(z)$ for $|z|<0.9$ to find \eqref{e:biascorr}.
\end{proof}

The logic of Lemma \ref{l:biascorr} is that if $G(\tau) = g(\tau)-g(0)$ is some error then one can estimate a hypothesized convergence rate $G(\tau) = O(\tau^a)$ by computing ${\dd \over \dd\log \tau} \log G$ and taking a sufficiently small $\tau$. However, if one uses $\tilde{G}(\tau) = g(\tau)-g(\tilde{\tau})$ to estimate $G(\tau)$, then the estimator $\tilde{a} = {\dd \over \dd\log\tau} \log\tilde{G}$ (for some given values of $\tau > \tilde{\tau}$) produces a biased estimate $\tilde{a}$. The bias factor $\beta$ is always larger than $1$, so $\tilde{a} > a$ for all sufficiently small $\tau$. Even worse, if $\tau \to \tilde{\tau}^+$, the bias factor becomes unbounded ($\beta \to \infty$), so that the estimate $\tilde{a}$ is ``infinitely optimistic''.

To estimate the convergence rate $a$ of our algorithm, we want to compute 
$$a = \lim_{\tau \to 0}{\dd \over \dd\log \tau} \log E(\tau).$$
Even if $E$ is differentiable, $\tilde{E}$ is unlikely to be differentiable because it is an average over finitely many noisy SPDE solutions, hence even the biased estimate ${\dd \over \dd\log \tau} \log \tilde{E}$ is not available. Thus, we used an averaged version of ${\dd \over \log \tau}\log \tilde{E}(\tau)$ by setting $\tilde{\tau} = 1/32$ and performing linear regression over $\tau=1/16,1/8,1/4,1/2$ to obtain approximately $\tilde{a} = 1.3 \pm 0.1$. 
Let $\beta_{\tau,\tilde{\tau}}(a)$ be given by \eqref{e:biascorr}. The average bias is $\beta_{\mu}(a) = {1 \over 4}(\beta_{{1 \over 2},{1 \over 32}}(a)+\ldots+\beta_{{1 \over 16},{1 \over 32}}(a)).$
We recover $a$ from $\tilde{a}$ by approximately solving \eqref{e:biascorr}, neglecting the $O(\tau)$ term. In other words, we (numerically) solve $a\beta_{\mu}(a) = \tilde{a} = 1.3 \pm 0.1$ for the unknown $a$, to obtain the bias-corrected estimate
$a \approx 0.88 \pm 0.14$ and $\alpha = 0.44 \pm 0.07$, which is entirely consistent with Theorem \ref{thm:4}.

{\small
\subsection*{Acknowledgments}
The authors wish to thank Lars Diening for helpful
discussions and comments. They are also grateful to the anomymous referee and the associate editor for the careful reading of the paper and the valuable suggestions.}


\end{document}